\documentclass[12pt,a4paper]{article}

\title{Real hyperbolic representations of   $PU(1,n)$ }
\author{Gonzalo Ruiz Stolowicz\thanks{EGG, Institute of Mathematics, EPFL, gonzalo.ruizstolowicz@epfl.ch}}
\usepackage[left=1.5in, right=1in, top=1in, bottom=1in]{geometry}                
\usepackage{amsthm}

\usepackage[utf8]{inputenc}
\usepackage{sectsty}   
\usepackage{amsfonts}
\usepackage{graphicx,import}
\usepackage{amssymb, amsmath}
\usepackage{tikz}
\usepackage{tikz-cd}
\usetikzlibrary{matrix,arrows,decorations.pathmorphing}
\usepackage{enumitem}
\usepackage{mathtools}
\usepackage{dsfont}
\usepackage{amsmath}
\usepackage{setspace}
\usepackage[capitalise]{cleveref}

\numberwithin{equation}{section}

\newtheorem{teo}{Theorem}[section]
\newtheorem{lem}[teo]{Lemma}
\newtheorem{prop}[teo]{Proposition}

\newtheorem*{teosin}{Theorem}
\theoremstyle{definition}

\newcommand{\N}{\mathbf N}

\newcommand{\R}{\mathbf R}

\newcommand{\C}{\mathbf C}

\newcommand{\hn}{{\mathbf H}^n}
\newcommand{\hi}{{\mathbf H}^\infty}
\newcommand{\iso}{\text{Isom}(\hi_\R)}

\makeatletter
\def\keywords{\xdef\@thefnmark{}\@footnotetext}
\makeatother
\begin{document}
	
	\maketitle
	
	\begin{abstract}
	It is shown that $PU(1,n)$, for $n\geq2$, does not admit non-elementary representations into the group of isometries of an infinite-dimensional real hyperbolic space. 
	\end{abstract}
	
	\section*{Introduction}
	The principal  contribution of this article is the following. 
	\begin{teosin}
		
		If $n\geq2$, the group of holomorphic isometries of $\hn_\C$,  the  complex  hyperbolic space of dimension n, does not admit non-elementary representations into the group of isometries of $\hi_\R $, a separable  infinite-dimensional  real hyperbolic space. 
	\end{teosin}
 Contrary to the finite-dimensional case, this statement is not an
instance of a general principle such as the Mostow-Karpelevich theorem.
Indeed, there do exist exotic non-elementary representations of the group of holomorphic isometries of $\hn_\C$, for every $n\geq1$, on the
infinite-dimensional complex hyperbolic space (see for example Theorem 1.15 in \cite{monod2018notes})  and likewise
from real hyperbolic to real hyperbolic (see for example Theorem B in \cite{monod2014exotic}). 

	The main idea used to prove  this appears in     \cite{CarlsonyToledo} and \cite{pydelzant}: the existence of smooth  harmonic maps  $\hn_\C\rightarrow \hi_\R$ associated to a uniform lattice and a non-elementary representation  of the group of isometries of the domain, together with the strong restrictions on the rank of such maps (see \cite{sampson}).
	
  	This paper is a follow up of \cite{monod2018notes} and  \cite{monod2014exotic}. In the first paper  the author studies general representations of groups into groups of isometries of hyperbolic spaces. In the second one, among other results, the authors classify the non-elementary representations of $PO(n,1)$ into the group of isometries of an infinite-dimensional real hyperbolic space. 
  	 In the process of trying to unify the two different perspectives of the aforementioned articles, the main question  addressed here arises naturally.

	\section{Preliminaries and notations.}	
	Suppose $\mathbf{F}=\R,\C$ and let $B$ be a non-degenerate bilinear form, $\mathbf{F}$-linear in the first argument and $\mathbf{F}$-antilinear in the second, defined on $H$, a Hilbert space  over $\mathbf{F}$.  Following \cite{burger2005equivariant}, the form $B$  is called \textit{strongly non-degenerate of signature} $(1,n)$, 
	with $n\in\N\cup\infty$, if $H$ admits a $B$-orthogonal decomposition $\mathbf{F}\oplus H'$, with $\dim_\mathbf{F} (H') =n$ and   where $B$ restricted to $\mathbf{F}$ is the usual inner product  and 
	$\left(H',-B|_{H'}\right)$ is a Hilbert space.  
	
	Let $(H,B)$ be a Hilbert space over $\mathbf{F}$ and let $B$ be a strongly non-degenerate bilinear form of signature $(1,n)$. For  $v\in H$ , define $[v]=\mathbf{F}v$. The $n$-$dimensional$ \textit{hyperbolic space over }$\mathbf{F}$,  associated to $(H,B)$, is defined as 
	$$\hn_{\mathbf{F}}=\{[v]\mid  B(v,v)>0\},$$ provided with the metric, 
	$$\cosh(d([v],[w]))=\frac{|B(v,w)|}{B(v,v)^{1/2}B(w,w)^{1/2}}.$$
	For   further reading about hyperbolic spaces of infinite dimensions see \cite{das2017geometry} and \cite{burger2005equivariant}. 
	
	Let $X$ be a metric space. Given three points $x,y,z\in X$  define the \textit{Gromov product} of $y$ and $z$ with respect to  $x$ as,   $$(y,z)_x=\frac{1}{2}\left(d(y,x)+d(z,x)-d(y,z)\right).$$ 
	A sequence  $(x_i)$ in $X$ is called a \textit{Gromov sequence} if for $z_0$ a (any) base point,  $$\lim_{n,m\to\infty}(x_n,x_m)_{z_0}=\infty.$$
	Two Gromov sequences, $(x_i)$ and $(y_i)$, are called \textit{equivalent} if for $z_0$ a (any) base point, 
	$$\lim_{n,m\to\infty}(x_n,y_m)_{z_0}=\infty.$$
	
	The relation defined above in the set of Gromov sequences is an equivalence relation. Denote $\partial_g X$ the set of equivalence classes of Gromov sequences in $X$. The set $\partial_g X$ will be called the \textit{boundary at infinity} of X.  
	
	In this work CAT(-1) spaces will always be considered complete. For definitions and an extensive study of these spaces and the definition of Gromov hyperbolicity see \cite{bridson2013metric}. 
	
	Every CAT(-1) space is hyperbolic in the sense of Gromov (Proposition 3.3.4 in \cite{das2017geometry}). Therefore there  are two natural ways to define and topologize a  {boundary at infinity} for a complete CAT(-1) space. The first one is considering $X$ as a Gromov hyperbolic space and taking $\partial_gX$. The second is considering $X$ as a CAT(0) space and defining the boundary at infinity as the set of equivalence classes of asymptotic geodesic rays.   It is a classical  result  that for a CAT(-1) space these two notions are equivalent. A sketch of  proof will be  given later due to the author's lack of knowledge of a reference in the literature. 
	
	This is  Lemma 3.4.10 in  \cite{das2017geometry}. 
	\begin{lem}\label{gromovproduct}
		Let $X$ is a $CAT(-1)$ space and suppose 
		$\xi,\eta\in\partial_gX$ and $z,w\in X$. If  $(x_i)\in\xi$ and $(y_i)\in\eta$, the limits 
		$$(\xi,\eta)_z=\lim_{n,m\to\infty}(x_n,y_m)_z$$
		and $$(\xi,w)_z=\lim_{n\to\infty}(x_n,w)_z$$ exist and do not depend on the choice of representatives. \end{lem}
	
Define $\mathcal{T}_g$ as the unique topology on $X\cup \partial_gX$ such that for $S\subset X\cup \partial_gX$, $S$ is open if, and only if, 
	$S\cap X$ is open for the metric topology and for every $\xi\in S\cap\partial_gX$, there exists $t\geq0$ such that $N_t(\xi)\subset S$, where 
	$$N_t(\xi)=\{y\in X\cup \partial_gX\mid (y,\xi)_{x_0}>t\}.$$
	
	The following is Lemma 3.4.22 in \cite{das2017geometry}. 
	\begin{prop}\label{continuidaddelproductodegromov}
		Let $X$ be a CAT(-1) space. Suppose $(z_n)$ is a sequence in $X$ and  suppose $(x_n)$ and $(y_n)$ are sequences in $X\cup\partial_gX$ converging with the topology $\mathcal{T}_g$ to $z\in X$ and $x,y\in X\cup\partial_gX$, respectively. Therefore  
		$$\lim_{n\to\infty}(x_n,y_n)_{z_n}=(x,y)_z.$$
	\end{prop}
	
	Let  $X$ be a complete CAT(0) space and $x_{0}\in X$  a base point. Given to geodesic rays $\sigma$ and $\tau $ issuing  from $x_0$, the map $t\mapsto d(\sigma(t),\tau(t))$ is a convex non-negative function that vanishes at $0$, therefore if it is bounded, then  it has to be constant. This observation gives sense to  the following definitions. 
	
	For $s>r$ there is a  projection $$\overline{B(x_0,s)}\xrightarrow{p_{r,s}}\overline{B(x_0,r)}.$$
	This defines and inverse system of topological spaces indexed by the positive numbers. Let 
	$$\overline{X}=\{[0,\infty)\xrightarrow{\sigma}X\mid \sigma(0)=x_0\,\, \text{and}\, \sigma \,\text{is a generalized geodesic ray}\}$$ be the inverse limit associated to this inverse system. 
	Here a generalized geodesic is either a geodesic ray issuing from  $x_0$ or a geodesic segment issuing from  $x_0$ defined in an interval $[0,r]$, which is considered constant  in $[r,\infty)$. 
	
	The topology of inverse limit in $\overline{X}$ (the subspace topology of the product $X^{\R_{\geq0}}$) is the same as the topology of uniform convergence in compact sets. This topology on $\overline{X}$, often called the \textit{cone topology}, and here denoted as $\mathcal{T}_c$, restricts to the metric topology on $X$ and it does not depend on the choice of the base point $x_0$ (see II.8.8 in \cite{bridson2013metric}). Denote  as  $\partial_cX$ the set of geodesic (infinite) rays with base point in $x_0$ with the topology of subspace of the cone topology.
	
	For every $r>0$ let $$\overline{X}\xrightarrow{p_r}\overline{B(x_0,r)}$$ be the function that is the identity in $\overline{B(x_0,r)}$ and $p_r(\sigma)=\sigma(r)$,  for any $\sigma $ generalized geodesic ray  that is not  constant  on $[r,\infty)$. 
	
	Given a geodesic ray $\xi$,    let $U(\xi,R,\epsilon)$
	be the set of generalized rays $\tau$ such that $\tau|_{[R,\infty )}$ is not constant and $d(p_R(\tau), p_R(\xi))<\epsilon.$ 
	Observe that given  a geodesic ray $\xi$,  the sets $U(\xi,R,\epsilon)$ are a neighborhood basis for the cone topology. 
	
	The following result is often called the finite approximation Lemma, see for example Theorem 1 in Chapter 8 of \cite{coornaert2006geometrie}. 
	\begin{lem}\label{finiteapprox}
		Suppose $(X,x_0)$ is a $\delta$-hyperbolic  geodesic space and  consider  $$\{x_1,\dots,x_n\}\subset X\cup\partial X.$$ Here a point at infinity is understood as the limit of  a geodesic ray.  Define   $Y$ as the  union of the geodesic segments or geodesic rays  $[x_0,x_i]$.   If $2n\leq 2^k+1$, there exists a simplicial tree $Tr(Y)$ and a map $Y\xrightarrow{f}Tr(Y)$ with the following properties:
		\begin{enumerate}
			\item For every $i$, the restriction of $f$ to $[x_0,x_i]$ is an isometry. 
			\item 
			For every $x,y\in Y$, 
			$$d(x,y)-2k\delta\leq d(f(x),f(y))\leq d(x,y).$$
		\end{enumerate} 
	\end{lem}
	When $n=2$ the tree of the finite approximation Lemma is a tripod where the extremes are $f(x_i)$, with $i=0,1,2$ (see Proposition 3.1 of Chapter 1 in \cite{coornaert2006geometrie}).
	
	As it was mentioned before, the following theorem is a classic result for which the author could not find a reference   in the literature for non-proper spaces. 	\begin{teo}
Let $X$ be a CAT(-1) space. There is a natural homeomorphism $$((X,\partial_c X), \mathcal{T}_c)\xrightarrow{\Psi}((X,\partial_g X), \mathcal{T}_g).$$	\end{teo}
	\begin{proof}
		Fix  a base point $z_0\in X$. Observe that for every geodesic ray $\tau$ with $\tau(0)=z_0$, the sequence $(\tau(t_n))$ is a Gromov sequence for any sequence $(t_n)\to\infty$ and the class of equivalence of this Gromov sequence does not depend on the choice of the sequence $(t_n)$.   
		Therefore for every  geodesic ray $\tau$ with starting point at $z_0$ there is a well defined Gromov sequence $[\tau]$. Let $\Psi$ be the  map such that  $\partial_cX\xrightarrow{}\partial_gX$ is 
		defined by  $\Psi(\sigma)=[\sigma]$ and the identity in $X$. 
		In Proposition 4 of Chapter 7 in \cite{ghysgroupes} the authors showed, for proper CAT(-1) spaces, that $\Psi|_{\partial_cX}$ is a bijection. The same proof can be applied in this context  if convergence arguments of Arzel\`{a}-Ascoli type are exchanged by properties of convergence of Gromov sequences and applications of the finite approximation Lemma. 
		
		The claim now is that $\Psi$ is a homeomorphism.  
		Fix $N_t([\sigma])$ for $t>0$ and a geodesic ray $\sigma $ issuing from  $z_0$. Call $C$ the general constant error coming from the tree approximation for $3$ points. Fix $R,\epsilon>0$ such that $R-\epsilon -C>t+1$. Let $\tau\neq\sigma $ be a geodesic ray from $x_0$ such that $d(\tau(R),\sigma(R))<\epsilon$ and consider any $s>R$.
		If $(\sigma(s),\tau(s))_{z_0}>t+1$, then $\sigma(s)\in N_t([\sigma])$. 
		If this is not the case, then $R>(\sigma(s),\tau(s))_{z_0}$ and from the tripod approximation for the points $\{z_0,\sigma(s),\tau(s)\}, $ 
		$$|(\sigma(s),\tau(s))_{z_0}-(\sigma(R),\tau(R))_{z_0}|<C.$$
		But $(\sigma(R),\tau(R))_{z_0}>R-\frac{\epsilon}{2},$
		therefore $(\sigma(s),\tau(s))_{x_0}>t+1, $ which is a contradiction. This shows that $(\sigma(s),\tau(s))_{x_0}>t+1$ and that 
		$$([\tau],[\sigma])_{x_0}=\lim_{s\to\infty}(\sigma(s),\tau(s))_{x_0}\geq t+1,$$
		or in other words, that  $[\tau]\in N_t([\sigma])$. To finish just observe that for every $r>0$, 
		$$\begin{array}{rcl}(\sigma(s+r),\tau(s))_{x_0}&=&\frac{1}{2}\big(2s+r-d(\sigma(s+r),\tau(s))\big)\\&\geq& \frac{1}{2}\big(2s-d(\sigma(s),\tau(s))\big)\\&=&(\sigma(s),\tau(s))_{z_0}. \end{array}$$
		This implies that
		$$\lim_{r\to\infty}(\sigma(s+r),\tau(s))_{x_0}=([\sigma],\tau(s))_{x_0}\geq t+1,$$
		that  shows $\Psi(U(\sigma, R,\epsilon ))\subset N_t([\sigma])$.
		
		Fix $R,\epsilon >0$ and consider $U(\sigma,R,\epsilon).$ Suppose that for every $t>0$  $$N_t([\sigma])\not\subset U(\sigma,R,\epsilon).$$ 
		Thus, for every $n\in\N$ there exists $x_n\in N_n([\sigma])\setminus U(\sigma,R,\epsilon). $ This means that for every $n$, $(x_n,[\sigma])_{z_0}\geq n$. 
		Choose $s_n$ such that for every $r\geq s_n$, $$(x_n,\sigma(r))_{z_0}\geq n.$$ Without lost of generality, suppose that $(s_n)_n$ and $(d(x_n,x_0))_n$ are increasing sequences. Using  the finite approximation lemma for 
		$$\{z_0,x_n,x_{n+r}, \sigma(s_{n+r})\},$$   it is possible to show that $(x_n)$ is a Gromov sequence.
				If $\sigma_n$ is the geodesic segment that connects $z_0$ to $x_n$, then 
		$$\gamma(t)=\lim_{n\to\infty }\sigma_n(t)$$ is a geodesic, in fact $\gamma$ is such that $\Psi(\gamma)=[(x_n)]$.  Here an abuse of notation is made because only for $n$ bigger than $t$ it is possible to assume that $\sigma_n(t)$ is defined.  By construction $d(\gamma(R),\sigma(R))\geq\epsilon$, therefore $\gamma\neq\sigma$, but this is a contradiction because  $(x_n)$  belongs to $[\gamma]$ and $[\sigma]$.  Therefore  there exists $t>0$ such that $$N_t([\sigma])\subset U(\sigma,R,\epsilon).$$ 
	\end{proof}
	\begin{lem}\label{distanciaalintervalo}
		If $X$ is a CAT(-1) space there exists a constant $C>0$ such that for  every $x,y ,z\in X$,
		$$|d(x,[y,z])- (x,y)_z |<C.$$	
	\end{lem}
	\begin{proof}
		This is just an  application of  Lemma \ref{finiteapprox} for   $w,x,y,z\in X$ where  $w\in[y,z]$  is the point that minimizes the distance between $x$ and the geodesic segment connecting $y$ and $z$. 
	\end{proof}
	
	In Theorem 1.1 of \cite{capracelytchak} the authors proved the main statement of the  following lemma  in a more general setting. Also in   Proposition 2.1 of  \cite{adams1997amenable} there is a similar result for locally compact CAT(0) spaces, using the idea of that proof, here an elementary argument is given.
	\begin{lem}\label{intersecciontheconvexos}
		Let $X$ be a CAT(-1) space and let $\mathcal{C}=\{C_i\}_{i\in\N}$ be a family of non-empty, closed and convex subsets of $X$ such that for every $n$, $C_{n+1}\subset C_n$. Suppose that  for some (any) $z_0\in X$,   $\lim_{n\to\infty}d(z_0,C_n)=\infty$, then   there exists $\xi\in\partial X$ such that, $$\{\xi\}= \bigcap_{n}\partial C_n.$$
	In particular  if there is a group $G$ acting by isometries  on $X$ and permuting the elements of $\mathcal{C}$, then $\xi$ is a G-fixed point.   
	\end{lem}
	
	\begin{proof}
		For every $n$ there is $x_n\in C_n$ such that $d(z_0,x_n)=d(z_0,C_n).$ There is a constant $C>0$ coming from the finite approximation lemma such that for every $n,m\in\N$, 
		$$|d(z_0,[x_n,x_m])-(x_n,x_m)_{z_0}|<C.$$ 
		If $m$ is bigger than $n$, 
		$$d(z_0,[x_n,x_m])\geq d(z_0,x_n), $$
		therefore $(x_n)$ is a Gromov sequence. If $\xi$ is   its  equivalence class,    then $\xi\in\bigcap_n\partial C_n$. 
		
		Suppose there is  $\eta\neq\xi $  such that $\eta\in\bigcap_n\partial C_n$. If  $\tau $ is  the unique  geodesic connecting  $\eta $ and $\xi$ (see Proposition 4.4.4 of \cite{das2017geometry}), then the image of $\tau$ is contained in every $C_n$, this is a contradiction because $\bigcap_nC_n=\emptyset$. 
		
		The last claim of the Lemma follows from the fact that $G$ also permutes the elements of $\{\partial C_n\}_{n}.$
	\end{proof}
	
	Let $G$ be a group acting on a space $X$. A function $X\xrightarrow{f}\R$ is called \textit{quasi-invariant} if for every $g$ there exists a constant $c(g)$ such that for every $x\in X$, 
	$$f(gx)-f(x)=c(g).$$
	
	Observe that the map $c$ in the previous definition has to be a homomorphism.  The statement of the next lemma, but in the context of  proper  CAT(0) spaces, appears  in Section 2 of  \cite{adams1997amenable}. The arguments there work also  for CAT(-1) spaces  given the statement of Lemma \refeq{intersecciontheconvexos} and the following observation.  Let  $\{C_i\}_{i\in\N}$ be a family of non-empty,  convex and closed sets in a complete CAT(0) space $X$ such  that  for every $n$, $C_{n+1}\subset C_{n}$. Therefore, $\bigcap_{n}C_n=\emptyset$ if, and only if, for every $x_0\in X$,  $\lim\limits_{n\to\infty}d(x_0,C_n)=\infty$ (see Proposition 1.2 of \cite{korevaari1997global}). 
	
	\begin{lem}\label{quasiinvariantfunctions}
		Let a group $G$ act by isometries on a CAT(-1) space $X$. If the action does not have  fixed points in $X\cup\partial X$, 
		then every quasi-invariant convex function defined on $X$ is $G$-invariant,  has a lower bound and the non-empty sublevel sets of it  are $G$-invariant and  unbounded.
	\end{lem}
	
	Let $G$ be a topological group and let $X$ be a topological space. An action of $G$ on $X$ is called \textit{orbitally continuous} if for every $x\in X$, the map $g\mapsto g\cdot x$ is continuous. If $X$ is a $CAT(-1)$ space an orbitally continuous representation $G\xrightarrow{\rho}Isom(X)$ is called \textit{non-elementary} if  it does not have finite orbits in $X\cup\partial X$. 	From now on all the representations will be considered orbitally continuous. 
	
	If $X$ is  $\text{CAT}(-1)$ space, $x_0$ is a base of point of $X$ and $\xi\in \partial  X$, the \textit{  Busemann function} based on $\xi$ and normalized in $x_0$ is defined as follows. If  $\sigma$ is the geodesic ray that starts at $x_0$ and points towards $\xi$, 
	$$b_{\xi,x_0}(y)=\lim_{t\to\infty}d(y,\sigma(t))-t.$$
	Observe that 
	$$\begin{array}{rcl}
		b_{\xi,x_0}(y)+2(y,\xi)_{x_0}&=&\lim\limits_{t\to\infty}( d(y,\sigma(t))-t
		) + \lim\limits_{t\to\infty}\big( d(y,x_0)+t-d(y,\sigma(t))
		\big) \\
		&=&d(y,x_0).\end{array}$$
	
	The following lemma is well known but a reference in the literature is unknown to the author. 
	\begin{lem}\label{orbitasfinitas}
		Let $X$ be a CAT(-1) space. A representation  $G\xrightarrow{\rho}Isom(X)$ is non-elementary if, and only if, it does not fix a point in $X\cup\partial X$ and  it does not  preserve a geodesic. 
	\end{lem}
	\begin{proof}
		Suppose that $\rho$ does not have fixed points in $X\cup \partial X$ and that it does not preserve a geodesic. 	If $\rho$ has a finite orbit in $X$, then it has a fixed point in $X$ (see  Corollary II.2.7 of \cite{bridson2013metric}).
		Suppose  that  there is $\{\xi_1,\dots,\xi_l\}$ a $G$-invariant set in $\partial X$ with $n\geq 3$. Fix a base point $x_0\in X$ and consider the function $f=\sum_{i=1}^{n}b_{\xi_i,x_0}$.  
		
		Observe that $b_{\xi,x_0}(gy)=b_{g^{-1}\xi,g^{-1}x_0}(y)$.  As the set $\{\xi_1,\dots,\xi_l\}$ is invariant, there is a permutation of $\{1,\dots,l\}$ defined by  $g^{-1}\xi_i=\xi_{\varphi(i)}$. Therefore 
		$$\begin{array}{rcl}b_{\xi,x_0}(gy)&=&b_{g^{-1}\xi_i,g^{-1}x_0}(y)\\
			&=&b_{\xi_{\varphi(i)},g^{-1}x_0}(y)\\
			&=&b_{\xi_{\varphi(i)},x_0}(y)-b_{\xi_{\varphi(i)},x_0}(g^{-1}x_0). 
		\end{array}$$
		As a consequence,  the convex function $f$ is quasi-invariant  because  
		$$f(gy)=\sum_{i=1}^{n}b_{\xi_i,x_0}(gy)=\sum_{i=1}^{n}b_{\xi_i,x_0}(y)-\sum_{i=1}^{n}b_{\xi_i,x_0}(g^{-1}x_0).$$
		
		By Lemma \ref{quasiinvariantfunctions}, any non-empty sublevel set of $f$ is unbounded. Fix one non-empty sublevel set $C_r$ and let $(y_n)\in C_r$ be an unbounded sequence. Up to taking a subsequence, we can suppose that $(y_n)$ converges to at most one point at infinity $\eta.$ Observe that 
		for every $\xi_i$, 
		$$b_{\xi_i,x_0}(y_n)=d(y_n,x_0)-2(y_n,\xi_i)_{x_0},$$
		therefore if $\eta\neq\xi_1,\dots,\xi_l$,  there exists $C>0$ such that for every $n$,  $$|f(y_n)-ld(y_n,x_0)|<C.$$ This is a contradiction because $\min(f)\leq f(y_n)\leq r$ and $\lim_n d(y_n,x_0)=\infty$. 
		
		Now suppose that $(y_n)$ converges to $\eta=\xi_1$. Observe that $b_{\xi_1,x_0}(y)\geq-d(y,x_0)$ and that there exists $C'>0$ such that for every $y_n$, 
		$$\begin{array}{rcl}f(y_n)&=&b_{\xi_1,x_0}(y_n) +b_{\xi_2,x_0}(y_n)+\cdots+b_{\xi_l,x_0}(y_n)\\&\geq& -d(y_n,x_0)+(l-1)d(y_n,x_0)-C'\\&\geq&(l-2)d(y_n,x_0)-C'.\end{array}$$
		Therefore $\{d(y_n,x_0)\}_n$ is bounded, which is a contradiction. 
	\end{proof}
	\section{The main result.}
	Let $G$ be a (Hausdorff) locally compact group. A discrete subgroup $\Gamma$ is called a \textit{lattice} if the space $G/H$ admits a non-zero  finite   $G$-invariant Radon measure. 
	
	The next proposition  appears in   Proposition  2.1 of \cite{caprace2009isometry} in the context of proper  CAT(0) spaces. The ideas in that article  can be used with slight modifications for the case of CAT(-1) spaces. 
	\begin{prop}\label{reticulanofijaunpunto}
		Suppose that $G$ is a locally compact and $\sigma$-compact group,   $\Gamma \leq G$ is a lattice and $X$ is a CAT(-1) space.  If $G\xrightarrow{\rho}Isom(X) $ is a non-elementary representation and $\rho|_\Gamma$ does not have fixed points in $X$, then $\rho|_{\Gamma}$ is a non-elementary representation. 
	\end{prop}
	\begin{proof}
		The proof will be by contradiction. Suppose that there exists  $\eta\in\partial X$ fixed by the action of  $\Gamma$. Using the continuous map $G/\Gamma\rightarrow\partial X$, induced by the orbit map $g\mapsto g\eta$, it is possible to define a  $G$-invariant probability   measure  $\mu$ in $\partial X.$ 
		Fix a point $x_0\in X$ and consider the function 
		$$F(y)=\int_{\partial X} b_{\xi,x_0}(y)d\mu(\xi)=\int_{G/\Gamma}b_{g \eta,x_0}(y)d\nu(g\Gamma),$$ 
		where $\nu$ is the $G$-invariant probability  measure in $G/\Gamma$.
		The function $\xi\mapsto  b_{\xi,x_0}(y)$ is continuous (see Lemma 3.4.22 in \cite{das2017geometry}) and for every $\xi\in\partial X$,  $|b_{\xi,x_0}(y)|\leq d(y,x_0)$.  This shows that the integral makes sense. 
		
		Every function $b_{\xi,x_0}$ is  convex, therefore  $F$ is convex too. Moreover, for every $g\in G$, 
		$$\begin{array}{rcl}
			F(g^{-1}y)&=&\int_{\partial X} b_{\xi,x_0}(g^{-1}y)d\mu(\xi)\\
			&=&\int_{\partial X} b_{g\xi,gx_0}(y)d\mu(\xi)\\
			&=&\int_{\partial X}\Big( b_{g\xi,x_0}(y)-b_{g\xi,gx_0}(x_0)\Big)d\mu(\xi)
			\\&=&
			\int_{\partial X}\Big( b_{g\xi,x_0}(y)-b_{\xi,x_0}(g^{-1}x_0)\Big)d\mu(\xi)
			\\&=&
			F(y)-F(g^{-1}x_0),\end{array} $$
		where the last equality holds because the  measure $\mu$ is $G$-invariant. 
		Therefore $F$ is quasi-invariant, and by Lemma \ref{quasiinvariantfunctions}, it is a $G$-invariant function. 
		
		Notice that $x_0\in C_0$,  the sublevel set of $F$ associated to  $0$. Observe that for every $n\in\N$ there exists $x_n\in C_0$ such that $d(x_0,x_n)>n$. Up to taking a subsequence, it is possible to suppose that $(x_n)$   converges  at most to $\xi_0\in \partial X$. 
		The claim is that $F(x_n)\to\infty$, which would be a contradiction. The proof for this statement  will follow  the ideas of Lemma 2.4 in \cite{burger1996cat}. 
		
		By   Lemma \ref{orbitasfinitas},  the orbit of every $\eta\in\partial X$ is infinite, hence $\mu$ is a non-atomic  measure, therefore		
		$$F(y)=\int\limits_{\partial X\setminus \xi_0}^{} b_{\xi,x_0}(y)d\mu(\xi).$$
		For every $y,z\in X$, 
		$$(y,z)_{x_0}\leq \min\{d(y,x_0), d(z,x_0)\},$$
		thus, for every $\eta \in\partial X$, 
		$(y,\eta)_{x_0}\leq d(y,x_0). $
		Therefore, for every $y\in X$ and $\eta\in\partial X$, 
		$$b_{\eta,x_0}(y)=d(y,x_0)-2(y,\eta)_{x_0}\geq -d(y,x_0).$$
		Define for every $n\in \N$ the measurable set 
		$$V(n)=\{\eta\in\partial X\mid \sup_{m\in\N}\{2(x_m,\eta)_{x_0}\}\leq n\}.$$
		The sequence $(x_n)$ belongs to at most $\xi_0$, therefore 
		$$\partial X\setminus\xi_0\subset \bigcup_{n}V(n). $$ 
	For every $n$,  $V(n)\subset V(n+1)$, thus there exists  some $n_0$ such that $\mu(V(n_0))>\frac{1}{2}.$
		Therefore for every  $x_m$, 
		$$\begin{array}{rcl}
			F(x_m)&=&\int\limits_{V(n_0)\setminus \xi_0} b_{\xi,x_0}(x_m)d\mu(\xi)+\int\limits_{(\partial X\setminus \xi_0)\setminus V(n_0)} b_{\xi,x_0}(x_m)d\mu(\xi)\\
			&\geq& \big(d(x_m,x_0)-n_0\big)\mu(V(n_0))-\big(1-\mu(V(n_0))\big)d(x_m,x_0)\\
			&=&\big(2\mu (V(n_0))-1\big)d(x_m,x_0)-n_0\mu(V(n_0)).
		\end{array}$$
		Thus 
		$F(x_m)\to\infty,$
		which is a contradiction.
		
		If $\rho|_\Gamma$ permutes two points at infinity there is an index two subgroup of $\Gamma$ that preserves a point at infinity. A finite index subgroup of a lattice is a lattice (see for example Lemma 1.6 in \cite{raghunathan1972discrete}), thus this assumption leads to a contradiction.
		
		If $\Gamma$ has a fixed point $x\in X$, the orbit map $g\mapsto g\cdot x$ induces in $X$ a $G$-invariant probability measure $\mu$. Consider  a nested family of  compact sets $\{K_i\}_{i\in\N}$ such that $\bigcup_{i}K_i=G$. 
		There exists $i$ such that $\mu(K_i\cdot x)>1/2$, therefore for every $g\in G$, 
		$$gK_i\cdot x \cap K_i\cdot x\neq\emptyset,$$ or in other words,  there are $k_1,k_2\in K_i$ such that $gk_1\cdot x=k_2 \cdot x$. 
		Observe that 
		$$\begin{array}{rcl}
			d(g\cdot x, x)&\leq &d(g\cdot x, gk_1\cdot x)+d(k_2\cdot x, x ).			
		\end{array}$$
		This shows that $x$ has a bounded orbit, but this is a contradiction because $G$ does not fix any point in $X$.  
	\end{proof}		
	Observe that the arguments in the previous proof show that if $G$ has a non-elementary representation on a CAT(-1) space $X$, then neither $\partial X$ nor $X$  admit a $G$-invariant probability measure. This property characterizes the  non-elementary representations. 
	\begin{prop}
		Let $G$ be a locally compact and $\sigma$-compact group and let $X$ be a CAT(-1) space. If   $G\xrightarrow{\rho} Isom(X)$ is a representation, then   $\rho$ is non-elementary if, and only if, neither $X$ nor $\partial X$   admit a $G$-invariant probability measure. 
	\end{prop}
\begin{proof}
	The implication that has  not been discussed   can be proved by considering Dirac masses. 
\end{proof}
	\begin{lem}\label{estanuniformementecercanas}
		Let $\Gamma_1$ and $\Gamma_2$ be two uniform  lattices of a locally compact group $G$ and let $X\xrightarrow{f_i}Y$, $i=1,2$ be two continuous functions between $X$ a topological space and $Y$ a metric space. Suppose $G$ acts  transitively on $X$ with compact stabilizers, by isometries on $Y$ and orbitally continuously  on both. If $f_i$ is $\Gamma_i$-equivariant,  then there exists $C>0$ such that for every $x\in X$, $d(f_1(x),f_2(x))<C.$ 
	\end{lem}
	\begin{proof}
		There exist  compact sets $K_i\subset G$  such that $\Gamma_i K_i=G$ (see for example Lemma 2.46 in \cite{folland1995course}).   Fix $x_0\in X$ and take $y \in X$. There exist $\gamma_i\in \Gamma_i$ and $k_i\in K_i$, such that $\gamma_ik_i x_0=y.$
		Therefore, 
		$$\begin{array}{rcl}
			d(f_1(y),f_2(y))&=&	d\big(\gamma_1f_1(k_1x_0), \gamma_2f_2(k_2x_0)\big)\\
			&=&d\big(\gamma_2^{-1}\gamma_1f_1(k_1x_0),f_2(k_2x_0) \big)\\
			&\leq&\sup\{d\big(zf_1(l_1x_0),f_2(l_2x_0)\big)\mid z\in K_2^{}\text{Stab}(x_0)K_1^{-1}, l_i\in K_i\}. 
		\end{array}$$
	\end{proof}
	
	\begin{lem}\label{suficientesisometrias}
		Let $\{H_n\}_{n\in\N_{\geq 1}}$ be a sequence of finite-dimensional hyperbolic spaces embedded in $\hi_\R$, where for $n\geq2$,  $H_n$ is isometric to  $\hn_\R$ and  $H_1$ is a geodesic.  Suppose that  for every $n\geq 1$,  $H_n\subset H_{n+1}$ and $$
		\overline{\bigcup_{n\geq 1 }H_n}=\hi_\R.$$ 
		Therefore, for every $n\geq2$ and $y_1,y_2\in\hi_\R$, there exists $\varphi\in\iso$ such that, $\varphi|_{H_n}=Id$ and $\varphi(\{y_1,y_2\})\subset H_{n+2}$.
	\end{lem}
	\begin{proof}
		Given $H_n\subset H_{n+2}$ and $y_1,y_2\in\hi_\R$, there exists $m\geq n+2$ and $\mathbf{H}$ isometric to $\mathbf{H}^m_\R$ such that $y_i\in \mathbf{H}$ and  $H_{n+2}\subset \mathbf{H}.$ Observe that every isometry of $\mathbf{H}$ can be extended to an isometry of $\hi_\R.$  Therefore the problem can be reduced to a statement about $\mathbf{H}^{m}_\R$, where the claim  is clear. 
	\end{proof}
	Let $M$ be a Riemannian manifold and let $U\subset M$ be an open set contained in a chart $(V,\phi)$. Suppose  that $\overline{U}\subset V$ and $\phi^{-1}(U)=B(p, r)\subset \R^m$. For every $\varphi\in \mathcal{C}^0(\partial U)$,  there exists a unique $h_\varphi\in \mathcal{C}^0(\overline{U})\cap \mathcal{C}^2( U)$ which solves the Dirichlet problem, in other
	words,  $h_\varphi$  is harmonic in $U$ ($\Delta h_\varphi|_U=0$) and $h_\varphi|_{\partial U}=\varphi|_{\partial U}$ (see Lemma 6.10 in \cite{Gilbargtrudingerelliptic}). For references about harmonic maps in the Riemannian setting  see \cite{nishikawa}  and for harmonic maps with a CAT(0) codomain see \cite{korevaarsobolev} and \cite{korevaari1997global}. 
	
	For every $x\in U$, the claim is that  the map $$\begin{array}{rcl}
		\mathcal{C}^0(\partial U)&\rightarrow &\R\\
		\varphi&\mapsto&h_\varphi(x)\end{array}$$ is a positive linear functional, in other words, it 
	defines a probability measure $p^U_x$ in $\partial U$. 
	Indeed, in  every $U$ as above, a harmonic map defined on $\overline{U}$ achieves its maximum (minimum) in $\partial U$ and   if there exists $u\in U$ such that the maximum (minimum) of $h$ is achieved in $u$, then $h$ is constant in $\overline{U}$  (see Theorem 3.1 in \cite{Gilbargtrudingerelliptic}). Thus  $$h_\varphi \leq \max\limits_{y\in\partial U} h_\varphi(y)=\max\limits_{y\in\partial U} \varphi(y),$$  therefore the linear map $\varphi\mapsto h_\varphi (x)$ is positive and  continuous for every $x\in U$. 
	
	A continuous function $M\xrightarrow{\varphi}\R$ is called  \textit{subharmonic} if for every $U$ as above and every $x\in U$, $$\varphi(x)\leq\int\limits_{\partial U} \varphi d p^U_x.$$
	
	If $\varphi \in \mathcal{C}^2(M)$, then $\varphi$ is subharmonic if, and only if, $\Delta f\geq 0$ (see page 103 of \cite{Gilbargtrudingerelliptic}). 
	
	Observe that every non-constant subharmonic function $\varphi $ defined on $\overline U$ satisfies a maximum principle: the maximum of $\varphi$ is achieved only in the boundary. 
	
	\begin{lem}\label{uniformconvergenceofsubharmonic}
		Let $M$ be a Riemannian manifold and  let $(\varphi _n)$ be a sequence of subharmonic functions defined in $M$.  If $(\varphi_n)\to \varphi$ uniformly on compact sets, then $\varphi$ is subharmonic. 
		
	\end{lem}
	
	The proof of the next lemma follows some of the ideas in  Theorem 2.3 in \cite{li1998harmonic}. 
	\begin{lem}\label{lemachinos}
		Let $X$ be a homogeneous and complete Riemannian manifold and  let   $$u,v :X\rightarrow \hi_\R$$ be two  harmonic and Lipschitz continuous functions of class $\mathcal{C}^2$. If there exists $C>0$  such that for every $x\in X$, $d(u(x),v(x))<C$, then either 
		$f=g$ or the images of $f$ and $g$ are contained in one geodesic.  
	\end{lem}
	\begin{proof}
		Suppose that $K>0$ is a Lipschitz constant for $u$ and $v$. 	Let $\{y_i\}_{i\in\N}\subset\hi_\R$ be such that if for every $n\geq1$, $H_n$ is  the smallest hyperbolic space that contains    $\{y_0,\dots,y_n\}$,  then the family 
		$\{H_n\}_{n\geq1}$  satisfies the hypothesis of  Lemma \ref{suficientesisometrias}. 
		
		Let $(x_n)_{n\geq1}$ be a sequence in $X$ such that 
		$$d(u,v)=\sup_{x\in X} \{d(u(x),v(x))\} =\lim_{n\to\infty}d(u(x_n),v(x_n)).$$
		
		Fix $x_0\in X$ and for every $i$ choose $\varphi_i\in\iso$ such that $\varphi_i(x_0)=x_i$. Define $u_i=u\circ\varphi_i$ and $v_i=v\circ\varphi_i$. 
		For every $i$ there exist an  isometry $T_{i}^1$ such that $T_i^1\circ u_i(x_0)=y_0$ and $T_i^1\circ v_i(x_0)\in H_1$. Observe that for every $i$, 
		$$d(T_i^1 \circ u_i(x_0),T_i^1 \circ v_i(x_0))\leq d(u,v).$$
		$H_1$ is locally compact,  therefore there exists a  subsequence $\big(T_{1,i}^1\circ v_{1,i}(x_0)\big)_{i\in\N}$ of $\big(T_i^1\circ v_i(x_0)\big)_{i\in\N}$ which is convergent. 
		
		Let $\{z_i\}_{i\in\N_{\geq1}}$ be a dense subset of $X$. Observe that for every $i$,  there exists an isometry  $T_i^2$ such that $T_i^2|_{H_1}=Id$ and  $$\{T_i^2\circ T_{1,i}^1\circ u_{1,i}(z_1), T_i^2\circ T_{1,i}^1\circ v_{1,i}(z_1)\}\subset H_3.$$
		Notice that for every $i$, 
		$$d\Big(T_i^2\circ T_{1,i}^1\circ u_{1,i}(z_1), T_i^2\circ T_{1,i}^1\circ u_{1,i}(x_0)\Big)\leq Kd(z_1,x_0), $$
		but $$T_i^2\circ T_{1,i}^1\circ u_{1,i}(x_0)=T_{1,i}^1\circ u_{1,i}(x_0)=y_0.$$
		
		Therefore $$\Big(T_i^2\circ T_{1,i}^1\circ u_{1,i}(z_1)\Big)_{i\in \N_{\geq1}}$$ is a bounded sequence in $H_3$. Also,  for every $i$, 
		$$d\Big(T_i^2\circ T_{1,i}^1\circ u_{1,i}(z_1), T_i^2\circ T_{1,i}^1\circ v_{1,i}(z_1)\Big)\leq d(u,v).$$
		Thus, 
		$$\Big(T_i^2\circ T_{1,i}^1\circ v_{1,i}(z_1)\Big)_{i\geq1}$$
		is again a bounded sequence in $H_3$. So it is possible to chose respective subsequences, 
		$$\Big(T_{2,i}^2\circ T_{2,i}^1\circ u_{2,i}(z_1)\Big)_{i{\geq1}}$$
		and 
		$$\Big(T_{2,i}^2\circ T_{2,i}^1\circ v_{2,i}(z_1)\Big)_{i{\geq1}} $$
		that are convergent. 
		
		By induction on $n$, suppose that for every for every $2\leq m\leq n$ and for every $i\geq1$ there are isometries $T^m_{m,i}$, and $T^1_{n,i}$ such that 
		\begin{enumerate}
			\item $T^1_{n,i}\circ u_i(x_0)=y_0$ and $\big(T^1_{n,i}\circ v_i(x_0)\big)_{i\geq1}$ is a convergent sequence in $H_1$. 
			\item 
			$T^m_{n,i}|_{H_{1+2(m-2)}}=Id. $
			\item 
			$\Big(T^m_{n,i}\circ\cdots\circ T^1_{n,i}\circ u_{n,i}(z_{m-1})\Big)_{i\geq1}$ and 	  $\Big(T^m_{n,i}\circ\cdots\circ T^1_{n,i}\circ v_{n,i}(z_{m-1})\Big)_{i\geq1}$ are converging sequences in $H_{1+2(m-1)}$. 
		\end{enumerate}
		For every $i\geq1$, let  $T^{n+1}_i$ be an isometry with the following properties,
		\begin{enumerate}
			\item $T^{n+1}_i|_{H_{1+2(n+1-2)}}=Id$. 
			\item $T^{n+1}_i\circ\cdots\circ T_{n,i}^1\circ u_{n,i}(z_n)$ and $T^{n+1}_i\circ\cdots\circ T_{n,i}^1\circ v_{n,i}(z_n)$ are elements of $H_{1+2(n+1-1)}$. 
		\end{enumerate}
		Observe that  
		$$\Big(T^{n+1}_i\circ\cdots\circ T_{n,i}^1\circ u_{n,i}(z_n)\Big)_{i\geq1}$$
		is a bounded sequence in $H_{1+2(n+1-1)}$, indeed 
		$$d\Big(T^{n+1}_i\circ\cdots\circ T_{n,i}^1\circ u_{n,i}(z_n)\,,\,T^{n+1}_i\circ\cdots\circ T_{n,i}^1\circ u_{n,i}(x_0) \Big)\leq Kd(z_n,x_0),$$
		but 
		$$T^{n+1}_i\circ\cdots\circ T_{n,i}^1\circ u_{n,i}(x_0) =y_0.$$
		Moreover, for every $i$, 
		$$ d\Big(T^{n+1}_i\circ\cdots\circ T_{n,i}^1\circ u_{n,i}(z_n)\,,\,T^{n+1}_i\circ\cdots\circ T_{n,i}^1\circ v_{n,i}(z_{n}) \Big)\leq d(u,v).$$
		Therefore 
		$$\Big(T^{n+1}_i\circ\cdots\circ T_{n,i}^1\circ u_{n,i}(z_n)\Big)_{i\geq1}$$
		and 
		$$\Big(T^{n+1}_i\circ\cdots\circ T_{n,i}^1\circ v_{n,i}(z_n)\Big)_{i\geq1}$$ are bounded sequences in $H_{1+2(n+1-1)}$. 
		Hence  it is possible to choose convergent subsequences 
		$$\Big(T^{n+1}_{n+1,i}\circ\cdots\circ T_{n+1,i}^1\circ u_{n+1,i}(z_n)\Big)_{i\geq1}$$
		and 
		$$\Big(T^{n+1}_{n+1,i}\circ\cdots\circ T_{n+1,i}^1\circ v_{n+1,i}(z_n)\Big)_{i\geq1}.$$
		
		Define now, 
		$$U(z_n)=\lim\limits_{i\to\infty}T^{i}_{i,i}\circ\cdots\circ T_{i,i}^1\circ u_{i,i}(z_n)$$
		and 
		$$ V(z_n)=\lim\limits_{i\to\infty}T^{i}_{i,i}\circ\cdots\circ T_{i,i}^1\circ v_{i,i}(z_n).$$
		Observe that there exists $M>0$ such that,  
		$$\begin{array}{rcl}U(z_n)&=&\lim\limits_{i\to\infty}T^{i}_{i,i}\circ\cdots\circ T_{i,i}^1\circ u_{i,i}(z_n)\\&=&\lim\limits_{i\to\infty}T^{M}_{i,i}\circ\cdots\circ T_{i,i}^1\circ u_{i,i}(z_n)\\
			&=&\lim\limits_{i\to\infty}T^{M}_{M,i}\circ\cdots\circ T_{M,i}^1\circ u_{M,i}(z_n)\\
		\end{array}$$
		and $$V(z_n)=\lim\limits_{i\to\infty}T^{M}_{M,i}\circ\cdots\circ T_{M,i}^1\circ v_{M,i}(z_n).$$
		Given $z_n$ and $z_m$, there exists $M'>0$ such that 
		$$\begin{array}{rcl}
			d(U(z_n),U(z_m))
			&=&\lim\limits_{i\to\infty}d(u\circ \varphi_{M',i}(z_n), u\circ \varphi_{M',i}(z_m))\\
			&\leq&Kd(z_n,z_m),
		\end{array}$$
		and with the same reasoning, 
		$$d(V(z_n),V(z_m))\leq K(d(z_n,z_m)).$$
		Therefore  $U$ and $V$ can be extended to $X$. 
		
		For every $m\geq 1 $, define $$R_m= T^{m}_{m,m}\circ\cdots\circ T^1_{m,m}\circ u_{m,m}$$ and 
		$$S_m=T^{m}_{m,m}\circ\cdots\circ T^1_{m,m}\circ v_{m,m}.$$
		Observe that for every $m$,  $R_m$ and $S_m$ are Lipschitz continuous functions  with Lipschitz constant smaller or equal than $K$. Therefore $\{R_n\}_n
		$ and $\{S_n\}_n
		$ are equicontinuous families. If the function $L_n$ is defined as  $L_n(z)=d(R_n(z),S_n(z)), $
		then the family  $\{L_n\}_n$ is   equicontinuous  and pointwise convergent to $z\mapsto d(U(z),V(z))$, thus by Arzel\`{a}-Ascoli Theorem, the convergence is uniform on compact sets. 
		
		The functions $u$ and $v$ are $\mathcal{C}^2$, and for every $i$, $\varphi_i $ is an isometry, therefore $u_i$ and $v_i$ are harmonic functions (see for example Proposition 2.2 in  \cite{ishi1979}). Moreover,  for every $i,j$, the map $T_{i,i}^j$ is an isometry, therefore  for every $m$, the functions $R_m$ and $S_m$ defined above are harmonic. For one reference for the last statement see the corollary at the end of page 131 of \cite{Eellssampson}.  
		
		The distance function $ \hi_\R\times \hi_\R\xrightarrow{d}\R$ is a (geodesically) convex function and for every $m$,  the map $x\mapsto d(R_m(x), S_m(x))$ is  harmonic (see the second example in page 133 of \cite{Eellssampson}). 
		Therefore, for every $m$ the function  $L_m$ is subharmonic (see Theorem 3.4 in \cite{ishi1979}) and by Lemma \ref{uniformconvergenceofsubharmonic}, the map $z\mapsto d(U(z),V(z))$ is subharmonic.
		
		Notice that for every $z\in X$  $d(u,v)\geq d(U(z),V(z))$, also  $$\begin{array}{rcl}
			d(U(x_0),V(x_0))&=&\lim\limits_{m}d(T_m(x_0),S_m(x_0))\\
			&=&\lim\limits_{m}d(u_{m,m}(x_0),v_{m,m}(x_0))\\&=&\lim\limits_{m}d(u(x_{m,m}),v(x_{m,m}))=d(u,v). 
		\end{array}$$
	Therefore $d(U(z),V(z))$ is constant as a consequence  of the maximum principle for subharmonic maps. By construction, for every $z$, $$d(U(z), V(z))=d(u(z), v(z)),$$ hence, by Lemma 2.2 in \cite{li1998harmonic}, either $u=v$ or the images of $u$ and $v$ are contained in a geodesic. 
	\end{proof}

	\begin{lem}
		If $\Gamma$ is a torsion free uniform lattice of $SU(1,n )$,  then the following hold:
		\begin{enumerate}	
			\item All the non-trivial elements act as hyperbolic isometries of $\hn_\C$. 
			\item If $\textit{l}(g)$ is the translation length of $g$ acting as an  isometry of $\hn_\C$, then $\inf\{\textit{l}(\gamma)\mid \gamma\in \Gamma\setminus e\}>0$.
			\item There exists $g\in SU(1,n)$ such that $g\Gamma g^{-1}$ and $\Gamma $ are non-commensurable. 
		\end{enumerate}
	\end{lem}
	\begin{proof}
		For 1) and 2) see Proposition II.6.10 in \cite{bridson2013metric} and  observe that if $g\in \Gamma \setminus e$ acts as an elliptic isometry, then it is contained in a compact (finite) subgroup of $\Gamma$ and this cannot be the case. 
		
		For 3) observe that every  $\gamma\in\Gamma\setminus e$ preserves a unique axis in $\hn_\C$ and that $\Gamma$ is finitely generated (see Theorem 6.15 and Remark 6.18 in \cite{raghunathan1972discrete}). Define 
		$$X=\{\xi\in\partial\hn_\C\mid \gamma\cdot \xi=\xi\, \text{for some}\,\, \gamma\in\Gamma\}.$$  Let $x\in X$ and $g\in  SU(1,n)$ be such that $g\cdot x \not\in X$. This is possible because $X$ is countable.   The claim is that $g\Gamma g^{-1}$ and $\Gamma$ are not commensurable. Indeed,  $g\cdot x$ is fixed  by some
		$\theta\in g\Gamma g^{-1}$, but for every $n$, $\theta $ and $\theta^n$ share the axis, therefore  the two lattices cannot be commensurable.  
	\end{proof}
	The existence of uniform lattices  in connected, non compact and semisimple groups is due to Borel, for one reference see Chapter XIV in \cite{raghunathan1972discrete}. Any of these lattices is finitely generated and as a consequence of Selberg's Lemma (see \cite{alperinselberg}) they are also virtually torsion free. This two facts together with the previous observations show that there exist $\Gamma_1$ and $ \Gamma_2$,  non-commensurable uniform  lattices in $SU(1,n)$. 

	Following  \cite{furstenbergnoteonborel}, a pair $(G,H)$ is called a \textit{Borel pair} if $G$ does not admit non-trivial homomorphisms to a compact group, $H$ is a closed subgroup and $G/H$ admits a finite $G$-invariant measure. 
	In this article  the author showed that if $(G,H)$ is a Borel pair, where $G$ is a connected   real algebraic group, then $H$ is Zariski dense in $G$ (see Corollary 4 in \cite{furstenbergnoteonborel}).
	
	\begin{lem}\label{generanalgodenso}Given two non-commensurable  lattices $\Gamma_1$ and $\Gamma_2$ of $SU(1,n)$ (or any connected real semisimple  linear algebraic group without compact factors), 
		the group $H$ generated by $\Gamma_1\cup \Gamma_2 $ is dense in $SU(1,n)$.\end{lem}
	
	\begin{proof} Observe that $\overline{H}$, the closure of $H$ for the usual topology,  is Zariski dense in $SU(1,n)$. Consider $\mathfrak{h}$ the Lie subalgebra of $\overline{H}$. This space is invariant under the action of ${H}$, therefore it is $SU(1,n)$-invariant because the action is Zariski continuous. This means that  $\overline{H}_0$ is a normal subgroup of $SU(1,n)$, but $SU(1,n)$ is simple. 
		 Suppose $\overline{H}_o$ is the  trivial group. 		
	Observe that $\overline{H}/\Gamma_i$ carries
		a finite invariant measure (see Lemma 1.6 in \cite{raghunathan1972discrete}), therefore  $\Gamma_1$ and $\Gamma_2$ have finite index in $\overline{H}$. This implies that $\Gamma_1$ and $ \Gamma_2 $ are commensurable, which is a contradiction. 
	\end{proof}

	Let $SU(1,n)\xrightarrow{\phi}Isom(\hn_\C)$ be the projectivization map.  This is a surjective homomorphism onto $Isom_\C(\hn_\C)$,  the group of holomorphic isometries of $\hn_c$. The map $\phi$ has finite kernel, therefore if $\Gamma_1$ and $\Gamma_2$ are as above, $\phi(\Gamma_1)$ and $\phi(\Gamma_2)$ are two uniform non-commensurable lattices of $Isom_\C(\hn_\C)$. Indeed, observe that $\Gamma_i\cdot \ker(\phi)$ is closed and countable (discrete), therefore there is $U$ an open subset of  $ SU(1,n)$ such that $U\cap(\Gamma_i\cdot ker(\phi) )=\{e\}.$ This shows that 
	$\phi(\Gamma_i)$ is a discrete subgroup of $Isom_\C(\hn_\C)$. 
	For the existence of a finite $\phi(SU(1,n))$-invariant measure observe that there is a natural continuous $G$-equivariant  bijection $$SU(1,n)/\Gamma_i\rightarrow \phi(SU(1,n))/\phi(\Gamma_i)$$ where the domain is compact. 
	The lattices $\phi({\Gamma_1})$ and $\phi(\Gamma_2)$ are not commensurable because $ker(\phi)$ is finite. The group generated by $\phi(\Gamma_1)$ and $\phi(\Gamma_2)$ is dense because $\Gamma_1$ and $\Gamma_2$ generate a dense subgroup of $SU(1,n)$.  
	
	\begin{teo}
		For $n\geq2, $ the group of holomorphic isometries of the complex hyperbolic space of dimension $n$,  $Isom_\C(\hn_\C)$, does not admit non-elementary representations into $Isom(\hi_\R)$, the group of isometries of the infinite-dimensional real hyperbolic space.
	\end{teo}
	\begin{proof}
		Let  $\rho$ is a non-elementary representation,  given a uniform lattice $\Gamma$ of $Isom_\C(\hn_\C)$, the restriction of $\rho$ to $\Gamma$ is non-elementary.  Therefore there  exists a $\Gamma$-equivariant, harmonic and Lipschitz continuous  map $\hn_\C\xrightarrow{u}\hi_\R$ (see Theorem 2.3.1 of \cite{korevaari1997global}). In Section 3.2 of \cite{pydelzant}, the authors showed that this map is  $\mathcal{C}^\infty$. 
		
		Given $\Gamma_1$ and $\Gamma_2$ two non-commensurable and uniform lattices of $Isom_\C(\hn_\C)$, there are   $\mathcal{C}^2$, harmonic, Lipschitz and $\Gamma_i$-equivariant functions, $\hn_\C\xrightarrow{u_i}\hi_\R$. Therefore it follows from Lemmas \ref{estanuniformementecercanas} and \ref{lemachinos} that $u_1=u_2$. This implies that the function $u=u_i$ is  $Isom_\C(\hn_\C)$-equivariant.	In Proposition 8 of  \cite{pydelzant}, the authors showed that the real rank of $u$ is at most 2. The arguments used there go back to the work of Sampson (see \cite{sampson}). If $x\in \hn_\C$, the kernel of $df_x$ is nontrivial. The group $Stab(x)$ acts transitively in spheres of the tangent space of $x$ and $u$ is $Isom_\C(\hn_\C)$-equivariant, therefore  $u$ is constant, but this is a contradiction. 
		
		\subsection*{Acknowledgments}
		I would like to thank Nicolas Monod for suggesting me the main question that this paper answers, as well as for the countless fruitful and enlightening discussions. Without his wisdom it would have been impossible for me to complete this text.  I would also like to thank Pierre Py for suggesting the way to solve this problem as well as for clarifying so many questions that arose in the process of creating this article. But above all for all his generosity during my stay in Strasbourg in the fall of 2021, without which this article would not have been possible. 
	\end{proof}
	
	\bibliographystyle{plain}
	\bibliography{biblio}

\begin{thebibliography}{10}

\bibitem{adams1997amenable}
Scot Adams and Werner Ballmann.
\newblock Amenable isometry groups of {H}adamard spaces.
\newblock {\em Mathematische Annalen}, 312(1):183--195, 1998.

\bibitem{alperinselberg}
Roger~C. Alperin.
\newblock An elementary account of {S}elberg's lemma.
\newblock {\em L'Enseignement Math\'{e}matique. Revue Internationale. 2e
  S\'{e}rie}, 33(3-4):269--273, 1987.

\bibitem{bridson2013metric}
Martin~R. Bridson and Andr\'{e} Haefliger.
\newblock {\em Metric spaces of non-positive curvature}, volume 319 of {\em
  Grundlehren der mathematischen Wissenschaften [Fundamental Principles of
  Mathematical Sciences]}.
\newblock Springer-Verlag, Berlin, 1999.

\bibitem{burger2005equivariant}
Marc Burger, Alessandra Iozzi, and Nicolas Monod.
\newblock Equivariant embeddings of trees into hyperbolic spaces.
\newblock {\em International Mathematics Research Notices}, (22):1331--1369,
  2005.

\bibitem{burger1996cat}
Marc Burger and Shahar Mozes.
\newblock {${\rm CAT}$}(-{$1$})-spaces, divergence groups and their
  commensurators.
\newblock {\em Journal of the American Mathematical Society}, 9(1):57--93,
  1996.

\bibitem{capracelytchak}
Pierre-Emmanuel Caprace and Alexander Lytchak.
\newblock At infinity of finite-dimensional {CAT}(0) spaces.
\newblock {\em Mathematische Annalen}, 346(1):1--21, 2010.

\bibitem{caprace2009isometry}
Pierre-Emmanuel Caprace and Nicolas Monod.
\newblock Isometry groups of non-positively curved spaces: discrete subgroups.
\newblock {\em Journal of Topology}, 2(4):701--746, 2009.

\bibitem{CarlsonyToledo}
James~A. Carlson and Domingo Toledo.
\newblock Harmonic mappings of {K}\"{a}hler manifolds to locally symmetric
  spaces.
\newblock {\em Institut des Hautes \'{E}tudes Scientifiques. Publications
  Math\'{e}matiques}, (69):173--201, 1989.

\bibitem{coornaert2006geometrie}
M.~Coornaert, T.~Delzant, and A.~Papadopoulos.
\newblock {\em G\'{e}om\'{e}trie et th\'{e}orie des groupes}, volume 1441 of
  {\em Lecture Notes in Mathematics}.
\newblock Springer-Verlag, Berlin, 1990.
\newblock Les groupes hyperboliques de Gromov. [Gromov hyperbolic groups], With
  an English summary.

\bibitem{das2017geometry}
Tushar Das, David Simmons, and Mariusz Urba\'{n}ski.
\newblock {\em Geometry and dynamics in {G}romov hyperbolic metric spaces},
  volume 218 of {\em Mathematical Surveys and Monographs}.
\newblock American Mathematical Society, Providence, RI, 2017.
\newblock With an emphasis on non-proper settings.

\bibitem{pydelzant}
Thomas Delzant and Pierre Py.
\newblock K\"{a}hler groups, real hyperbolic spaces and the {C}remona group.
\newblock {\em Compositio Mathematica}, 148(1):153--184, 2012.

\bibitem{Eellssampson}
James Eells, Jr. and J.~H. Sampson.
\newblock Harmonic mappings of {R}iemannian manifolds.
\newblock {\em American Journal of Mathematics}, 86:109--160, 1964.

\bibitem{folland1995course}
Gerald~B. Folland.
\newblock {\em A course in abstract harmonic analysis}.
\newblock Studies in Advanced Mathematics. CRC Press, Boca Raton, FL, 1995.

\bibitem{furstenbergnoteonborel}
Harry Furstenberg.
\newblock A note on {B}orel's density theorem.
\newblock {\em Proceedings of the American Mathematical Society},
  55(1):209--212, 1976.

\bibitem{ghysgroupes}
\'{E}tienne Ghys and Pierre de~la Harpe.
\newblock In {\em Sur les groupes hyperboliques d'apr\`es {M}ikhael {G}romov
  ({B}ern, 1988)}, volume~83 of {\em Progr. Math.}, pages 1--25. Birkh\"{a}user
  Boston, Boston, MA, 1990.

\bibitem{Gilbargtrudingerelliptic}
David Gilbarg and Neil~S. Trudinger.
\newblock {\em Elliptic partial differential equations of second order}.
\newblock Classics in Mathematics. Springer-Verlag, Berlin, 2001.
\newblock Reprint of the 1998 edition.

\bibitem{ishi1979}
T\^{o}ru Ishihara.
\newblock A mapping of {R}iemannian manifolds which preserves harmonic
  functions.
\newblock {\em Journal of Mathematics of Kyoto University}, 19(2):215--229,
  1979.

\bibitem{korevaarsobolev}
Nicholas~J. Korevaar and Richard~M. Schoen.
\newblock Sobolev spaces and harmonic maps for metric space targets.
\newblock {\em Communications in Analysis and Geometry}, 1(3-4):561--659, 1993.

\bibitem{korevaari1997global}
Nicholas~J. Korevaar and Richard~M. Schoen.
\newblock Global existence theorems for harmonic maps to non-locally compact
  spaces.
\newblock {\em Communications in Analysis and Geometry}, 5(2):333--387, 1997.

\bibitem{li1998harmonic}
Peter Li and Jiaping Wang.
\newblock Harmonic rough isometries into {H}adamard space.
\newblock {\em Asian Journal of Mathematics}, 2(3):419--442, 1998.

\bibitem{monod2018notes}
Nicolas Monod.
\newblock Notes on functions of hyperbolic type.
\newblock {\em Bulletin of the Belgian Mathematical Society. Simon Stevin},
  27(2):167--202, 2020.

\bibitem{monod2014exotic}
Nicolas Monod and Pierre Py.
\newblock An exotic deformation of the hyperbolic space.
\newblock {\em American Journal of Mathematics}, 136(5):1249--1299, 2014.

\bibitem{nishikawa}
Seiki Nishikawa.
\newblock {\em Variational problems in geometry}, volume 205 of {\em
  Translations of Mathematical Monographs}.
\newblock American Mathematical Society, Providence, RI, 2002.
\newblock Translated from the 1998 Japanese original by Kinetsu Abe, Iwanami
  Series in Modern Mathematics.

\bibitem{raghunathan1972discrete}
M.~S. Raghunathan.
\newblock {\em Discrete subgroups of {L}ie groups}.
\newblock Ergebnisse der Mathematik und ihrer Grenzgebiete, Band 68.
  Springer-Verlag, New York-Heidelberg, 1972.

\bibitem{sampson}
J.~H. Sampson.
\newblock Applications of harmonic maps to {K}\"{a}hler geometry.
\newblock In {\em Complex differential geometry and nonlinear differential
  equations ({B}runswick, {M}aine, 1984)}, volume~49 of {\em Contemp. Math.},
  pages 125--134. Amer. Math. Soc., Providence, RI, 1986.

\end{thebibliography}
	
\end{document}